\documentclass[11pt]{article}
\usepackage{amsmath,amssymb,amsfonts,amsthm}
\usepackage{graphicx}
\newtheorem{theorem}{Theorem}
\numberwithin{theorem}{section}
\newtheorem{corollary}[theorem]{Corollary}
\newtheorem{lemma}[theorem]{Lemma}
\newtheorem{proposition}[theorem]{Proposition}

\theoremstyle{remark}

\newcommand{\R}{\mathbb{R}}
\newcommand{\C}{\mathbb{C}}

\newcommand{\Z}{\mathbb{Z}}

\newcommand{\HH}{\mathbb{H}}
\newcommand{\E}{\mathbb{E}}
\newcommand{\Gamm}{\mathcal C}
\def \interior {{\hbox {int}}}

\def\H{\mathbb{H}}
\def\U{\mathbb{U}}
\def\diam{\mathop{\mathrm{diam}}}

\def\Im{{\rm Im}\,}

\def\SLEkk#1/{$\mathrm{SLE}_{#1}$}
\def\SLEk/{\SLEkk{\kappa}/}
\def\SLEtwo/{\SLEkk2/}
\def\SLE/{$\mathrm{SLE}$}
\def\CLEkk#1/{$\mathrm{CLE}_{#1}$}
\def\CLEk/{\CLEkk{\kappa}/}
\def\CLEtwo/{\CLEkk2/}
\def\CLE/{$\mathrm{CLE}$}
\def\GLEkk#1/{$\mathrm{GLE}_{#1}$}
\def\GLEk/{\GLEkk{\kappa}/}
\def\GLEtwo/{\GLEkk2/}
\def\GLE/{$\mathrm{GLE}$}
\def\SLEkr/{$\mathrm{SLE}_{\kappa, \rho}$}

\def\Ito/{It\^o}
\def \eps {\varepsilon}
\def \P {{\bf P}}

\def \E {{\bf E}}

\def\hcap{\mathrm{hcap}}

\def\diam{\mathrm{diam}}

\title{Conformal loop ensembles: Construction via loop-soups}
\author{{\sc Scott Sheffield\thanks{Massachusetts Institute of Technology. Partially supported by NSF grants DMS 0403182 and DMS 0645585 and
OISE 0730136.}}\, and
{\sc Wendelin Werner\thanks{Universit\'e Paris-Sud 11 and Ecole Normale Sup\'erieure. Research supported in part by ANR-06-BLAN-00058}
}
}
\date{}

\begin{document}
\maketitle
\begin{abstract}
The two-dimensional Brownian loop-soup is a Poissonian random collection of loops in a planar domain with an intensity parameter $c$.  When $c \leq 1$, we show that the outer boundaries of the loop clusters are disjoint simple loops (when $c>1$ there is a.s.\ only one cluster) that satisfy certain conformal restriction axioms.  We prove various results about loop-soups, cluster sizes, and the $c=1$ phase transition.


Combining this with the results of another paper of ours on the Markovian characterization of simple {\em conformal loop ensembles} (CLE), this proves that these outer boundaries of clusters of Brownian loops are in fact SLE$_\kappa$ loops
for $\kappa$ in  $(8/3, 4]$. More generally, it
completes the proof of the fact that the following three descriptions of simple CLEs (proposed in earlier works by the authors)
are equivalent:
\begin{enumerate}
\item The random loop ensembles traced by branching Schramm-Loewner Evolution (SLE$_\kappa$) curves for $\kappa$ in $(8/3, 4]$.
\item The outer-cluster-boundary ensembles of Brownian loop-soups for $c \leq 1$.
\item The (only) random loop ensembles satisfying the {\em conformal restriction axioms}.
\end{enumerate}
\end{abstract}

\section{Introduction}
\subsection{Overview}
What are the natural ``random loop ensemble'' analogs of the Schramm-Loewner evolution SLE$_\kappa$ when the paths are simple (i.e., $\kappa \leq 4$)?
This is a natural question, as such families of loops should describe
the scaling limits of the collection of interface loops appearing in various critical two-dimensional lattice models in statistical physics.

One way to address the above question is to use the SLE curves themselves.  In this approach, which appears in \cite{Sh}, one chooses a point $R$ on the boundary of a simply connected domain $D$ and launches from $R$ a branching SLE$_{\kappa, \kappa-6}$ process (for
$\kappa$ in $(8/3,4]$). The SLE$_{\kappa, \kappa -6}$ process is a natural variant of SLE$_\kappa$ with a target-independence property (see \cite {SchW}) that allows one to construct a branching tree of such processes in a canonical way. Using this ``exploration tree'', one can define a random collection of loops, which was called a {\em conformal loop ensemble} with parameter $\kappa$ (CLE$_\kappa$) in \cite {Sh}. These loop ensembles are the conjectured scaling limits for various random
models from statistical physics.  For example, CLE$_3$ appears as the scaling limit of spin cluster boundaries
of the critical Ising model (see \cite {Sm, ChSm}).

We note that some features of these CLE$_\kappa$ loop-ensembles are not easy to derive via this branching SLE
 definition. For instance, the general question of whether their law is independent of the choice of the starting point $R$ (the ``root'' of the branching-tree)
i.e., whether these CLE$_\kappa$ are invariant under all conformal maps from $D$ onto itself, remained open.

A completely different approach, proposed in \cite {MR2023758}, is based on the Brownian loop-soup. The Brownian loop-soup is a random Poissonian collection of Brownian loops in a domain $D$ that inherits nice properties from Brownian motion, including conformal invariance. It was introduced in \cite {LW}, motivated by questions concerning restriction properties of SLE processes in \cite {LSWr}. In a loop-soup, the Brownian loops can intersect and therefore create clusters. The outer boundaries of these clusters are random conformally invariant curves.  The way the laws of these curves vary when the boundary of the domain changes is related to the way the laws of SLE curves vary (these are the
``restriction properties'' of the curves), and it is therefore natural to think they are actually SLE-type loops.

The authors of the present paper recently showed in \cite{sheffieldwerner1}
that any random collection of loops satisfying the so-called {\em conformal restriction axioms} must be one of the CLE$_\kappa$ defined via the branching
SLE$_{\kappa, \kappa-6}$ tree with $\kappa$ in $(8/3,4]$ (and that it could also be constructed via a Poisson point process of SLE$_\kappa$ excursions).  In that paper, the term CLE was used to refer to any simple loop ensemble satisfying the conformal restriction axioms.  There, as here, the scope was limited to simple loop ensembles (where loops are disjoint and simple loops).
 It was not shown in either \cite{sheffieldwerner1} or \cite{Sh} (and remained far from obvious) that the CLE$_\kappa$ actually satisfy the axioms (in part because of the starting-point independence problem). In other words, it was proved that {\em if} any CLE exist, then they must have the form CLE$_\kappa$ for some $\kappa$ in $(8/3,4]$, but it was not proved that any CLE exist, so the result was a rather conditional one. However, it becomes very useful when combined with the results of the present paper.

Here, we will study properties of the Brownian loop-soup clusters. We will describe properties of its ``subcritical phase'' and of its phase transition:
a property of the decay rate of the size of clusters at the critical point, and the fact that the collection of ``loop-soup cluster outer boundaries'' satisfies the conformal restriction axioms introduced in \cite{sheffieldwerner1}.
This  will allow us to apply the results of that paper to deduce that these loup-soup boundaries are CLE$_\kappa$'s. Furthermore, this enables us to identify precisely the phase transition point of the loop-soup percolation, and to show that all CLE$_\kappa$'s for $\kappa \in (8/3, 4]$ can be constructed via Brownian loop-soups.
This therefore completes the program of \cite{MR2023758, Sh, sheffieldwerner1}, and in particular implies that the CLE$_\kappa$ {\em do} in fact satisfy the conformal restriction axioms and that they are root-independent for all $\kappa$ in $(8/3, 4]$.

Hence, we now have two remarkably different explicit constructions of these conformal loop ensembles CLE$_\kappa$ for each $\kappa$ in $(8/3,4]$.
This is useful, since many properties that seem very mysterious from one perspective are easy from the other.  For example, the (expectation) fractal dimensions of the individual loops and of the set of points not surrounded by any loop can be explicitly computed with SLE tools \cite{MR2491617}, while many monotonicity results and FKG-type correlation inequalities are immediate from the loop-soup construction \cite{MR2023758}. One illustration of the interplay between these two approaches is already present in this paper: One can use SLE tools to determine exactly the value of the critical intensity that separates the two percolative phases of the Brownian loop-soup (and to our knowledge, this is the only self-similar percolation model where this is the case).

Another consequence of this identity between CLE$_\kappa$ and Brownian loop soup cluster boundaries, is the ``additivity property'' of CLE's that we will describe in more detail later in the introduction: The union of two independent CLE's defines another CLE.

The introductions of the papers \cite{MR2023758, Sh, sheffieldwerner1} contain additional history and motivation, and we refer the interested reader to them.
Let us also mention that the present paper gives an interpretation of the {\em central charge} $c$ of models (as usually defined in theoretical physics) for $c$ in
$(0,1]$
as the intensity of the underlying Brownian loop-soup.

\subsection{Notation and axioms}

\subsubsection {Set-up}
The {Brownian loop-soup}, as defined in \cite{LW}, is a Poissonian random countable collection
of Brownian loops contained within a fixed simply-connected domain $D$.  We will actually only need to
consider the outer boundaries of the Brownian loops, so we will take the perspective that a loop-soup is
a random countable collection of simple loops (outer boudaries of Brownian loops
 can be defined as SLE$_{8/3}$ loops, see \cite {MR2350053}).
Let us stress that the {conformal loop ensembles} CLE$_\kappa$, with $\kappa \in (8/3,4]$,
are also random collections of simple loops, but that, unlike the loops of the Brownian loop-soup, the
loops in a CLE are almost surely disjoint from one another.

Random collections of loops can be formally defined using the notational framework of \cite{sheffieldwerner1},
which we briefly review here:  a { simple loop} in the complex plane $\C$ is the image of the unit circle in $\C$ under a continuous injective map. A simple loop $\gamma$ separates $\C$ into two connected components that we call its interior $\interior (\gamma)$ (the bounded one) and its exterior (the unbounded one). We endow the space ${\cal L}$ of simple loops with the $\sigma$-field $\Sigma$ generated by all the events of the type
$ \{ O \subset \interior (\gamma) \}$ when $O$ spans the set of open sets in the unit plane. Note that this $\sigma$-field is also generated by the events of the type $\{ x \in \interior ( \gamma ) \}$ where $x$ spans a countable dense subset $Q$ of the plane (recall that we are considering simple loops so that $O \subset \interior ( \gamma)$ as soon as $O \cap Q \subset \interior (\gamma)$).  It is not hard to see that
$\Sigma$ is equivalent to the Borel $\sigma$-algebra of the Hausdorff metric on ${\cal L}$.  Thus, the following extension is also natural:
let $\overline {\cal L}$ be the set of {\em all} closed bounded subsets of $\C$ and $\overline \Sigma$ the Borel $\sigma$-algebra of the Hausdorff metric on $\overline {\cal L}$.

A countable collection $\Gamma= (\gamma_j , j \in J)$ of simple loops can be identified with a
point-measure $\mu_\Gamma= \sum_{j \in J} \delta_{\gamma_j}$.
The space $\Omega$ of countable collections of elements of $\mathcal L$ is naturally equipped with the $\sigma$-field $\mathcal F$
generated by the sets $\{ \Gamma \ : \ \# ( \Gamma \cap A) = k \} = \{ \Gamma \ : \ \mu_\Gamma (A) = k \}$, where $A \in \Sigma$.
Similarly, there is an obvious extension: the space $\overline \Omega$ of countable collections of elements of $\overline {\mathcal L}$  is naturally equipped with the $\sigma$-field $\overline {\mathcal F}$
generated by the sets $\{ \Gamma \ : \ \# ( \Gamma \cap A) = k \} = \{ \Gamma \ : \ \mu_\Gamma (A) = k \}$, where $A \in \overline \Sigma$.

\subsubsection{Brownian loop-soups}
The loops of the {\em Brownian loop-soup}  $\Gamma = (\gamma_j, j \in J)$ in the unit disk $\U$
 are the points of a Poisson point process with intensity
$c \mu$, where $c$ is an {\em intensity} constant, and $\mu$ is the {\em Brownian loop measure in $\U$} on $({\cal L}, \Sigma)$.
The Brownian loop-soup measure $\mathbb P = \mathbb P_c$
is the corresponding probability measure on $(\Omega, \mathcal F)$.
When $A$ and $A'$ are two closed bounded subsets of a bounded domain $D$, we denote by
$L(A,A'; D)$ the $\mu$-mass of the set of loops that intersect both sets $A$ and $A'$, and stay in $D$.
When the distance between $A$ and $A'$ is positive, this mass is finite \cite{LW}.
Similarly, for each fixed positive $\epsilon$, the set of loops that stay in the bounded domain $D$ and have diameter
larger than $\epsilon$, has finite mass for $\mu$.

The conformal restriction property of the Brownian loop measure $\mu$ (which in fact characterizes the measure up to a multiplicative constant; see \cite{MR2350053}) implies the following two facts (which are essentially the only features of the Brownian loop-soup that we shall use in the present paper):
\begin{enumerate}
\item Conformal invariance: The measure $\mathbb P_c$ is invariant under any Moebius transformation of the unit disc onto itself. This invariance makes it in fact possible to define the law
$\mathbb P_D$ of the loop-soup in any simply connected domain $D \not= \C$ as the law of the image of $\Gamma$ under any given conformal map $\Phi$ from $\U$ onto $D$ (because the law of this image does not depend on the actual choice of $\Phi$).
\item Restriction: If one restricts a loop-soup in $\U$ to those loops that stay in
a simply connected domain $U \subset \U$, one gets a sample of $\mathbb P_U$.
\end{enumerate}
We will work with the usual definition (i.e., the usual normalization) of the measure $\mu$ (as in \cite {LW} --- note that there is often some confusion about a factor $2$ in the definition, related to whether one keeps track of the orientation of the Brownian loops or not). Since we will be talking about some explicit values of $c$ later, it is important to specify this normalization. For a direct definition of the measure $\mu$ in terms of Brownian loops, see \cite {LW}.

\vskip 5cm

\begin {figure}[htbp]
\hskip 4cm
\includegraphics [width=1.6in]{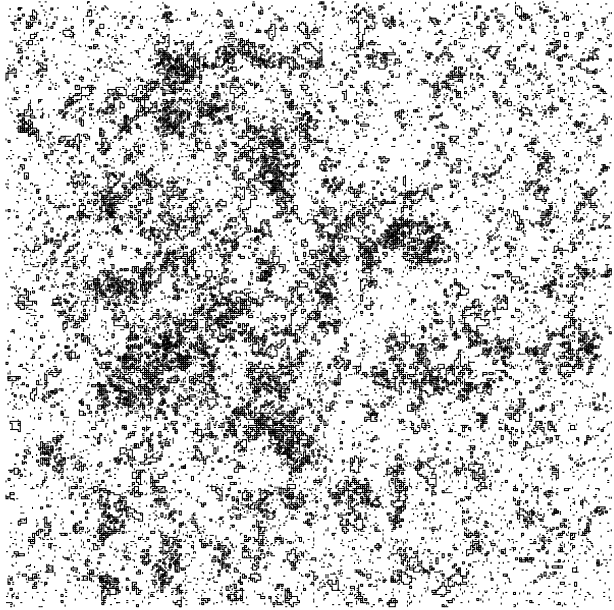}
\caption {Sample of a random-walk loop-soup approximation \cite {LTF} of a Brownian loop-soup in a square, by Serban Nacu}
\end {figure}

\subsubsection{Conformal restriction axioms}
We now describe the axioms introduced in \cite{sheffieldwerner1}.
We say that $\overline \Gamma = (\gamma_j , j \in J)$ is a simple loop configuration in the unit disc $\U$ if
it is a collection of simple loops in $\U$ that are {\em disjoint}
(for all $j \not= j' \in J$, $\gamma_j \cap  \gamma_{j'} = \emptyset$) and {\em non-nested} (for all $j \not= j'$ in $J$,
 $\gamma_j \not \subset \interior(\gamma_{j'})$) and that is {\em locally finite} (for each $\eps >0$, only finitely many loops $\gamma_j$ have a diameter greater than $\eps$).

We say that a non-empty random simple loop configuration $\overline \Gamma = ( \gamma_j , j \in J)$
 in the unit disc $\U$ (or rather its law $P_\U$) satisfies the {\em conformal
 restriction axioms} (introduced in \cite{sheffieldwerner1}) if the following hold:
\begin{enumerate}
\item {Conformal invariance:}  The law of $\overline \Gamma$ is invariant under any conformal transformation from $\U$ onto itself. As for the loop-soup,
this enables us to define the law $P_D$ of the loop-ensemble in any simply connected domain $D \not= \C$. We can also define $P_D$ if $D$ has several connected components by taking independent samples in each of the connected components.
\item {Restriction:} We want to describe the law of the loops on the loop-ensemble that stay in a simply connected subset $U$ of $\U$. To state this natural property for random families $\overline \Gamma$ of {\em disjoint loops}, we need to introduce some notation. Define $U^*$ to be the set obtained by removing from $U$ all the loops and interiors of
loops of $\overline \Gamma$ that do not stay in $U$.
This set $ U^*$ is an open subset of $\U$ (because of the local finiteness condition). The restriction property is that (for all $U$), given $U^*$,
the conditional law of the set of loops in $\overline \Gamma$ that {\em do} stay in $U$ (and therefore in $U^*$) is $P_{U^*}$.
\begin {figure}[htbp]
\begin {center}
\includegraphics [width=2in]{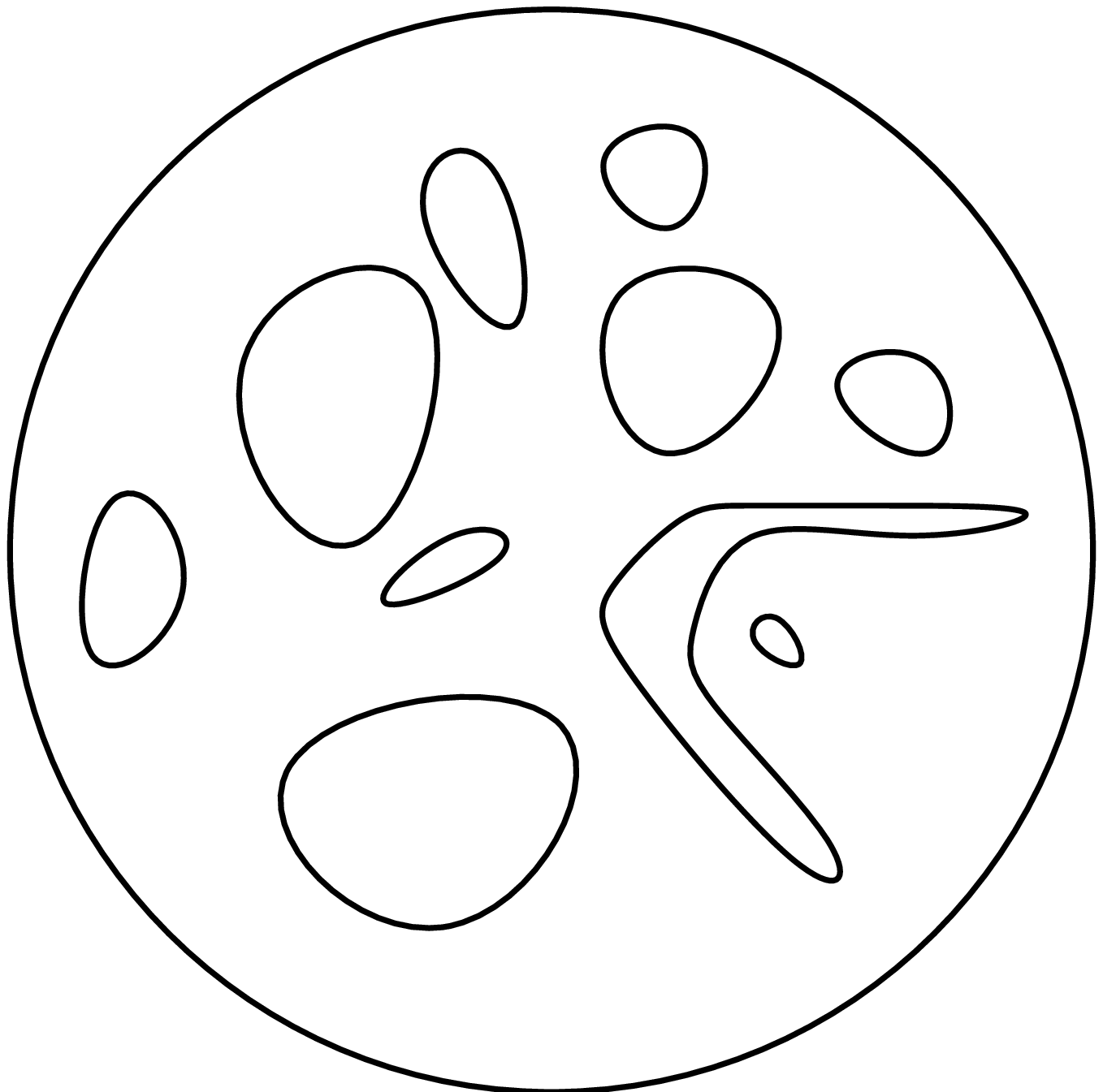}
\includegraphics [width=2in]{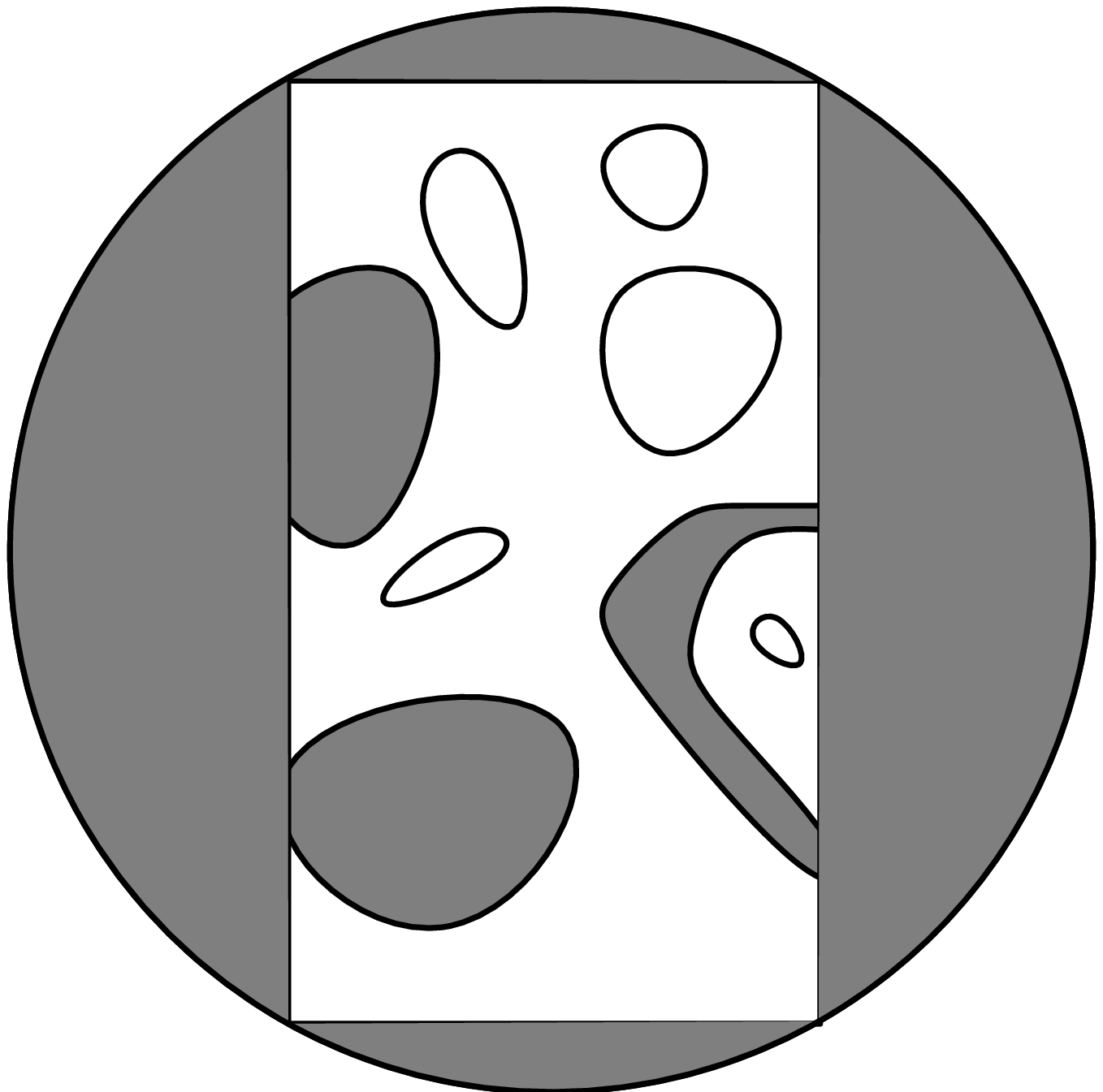}
\caption {CLE restriction property (sketch): $U$ is the rectangle and $\U \setminus U^*$ is darkened}
\end {center}
\end {figure}
\end{enumerate}

\subsection{Main results}
As mentioned above, \cite{MR2023758} pointed out a  way to relate Brownian loop-soups clusters to SLE-type loops:
Two loops in $\Gamma$ are said to be adjacent if they intersect.  Denote
by $\mathcal C(\Gamma)$ the set of clusters of loops under this relation.  For each element
$C \in \mathcal C(\Gamma)$ write $\overline C$ for the closure of the union of all the loops in $C$ and denote
by $\overline \Gamm$ the family of all $\overline C$'s.

\begin {figure}[htbp]
\begin {center}
\includegraphics [width=2in]{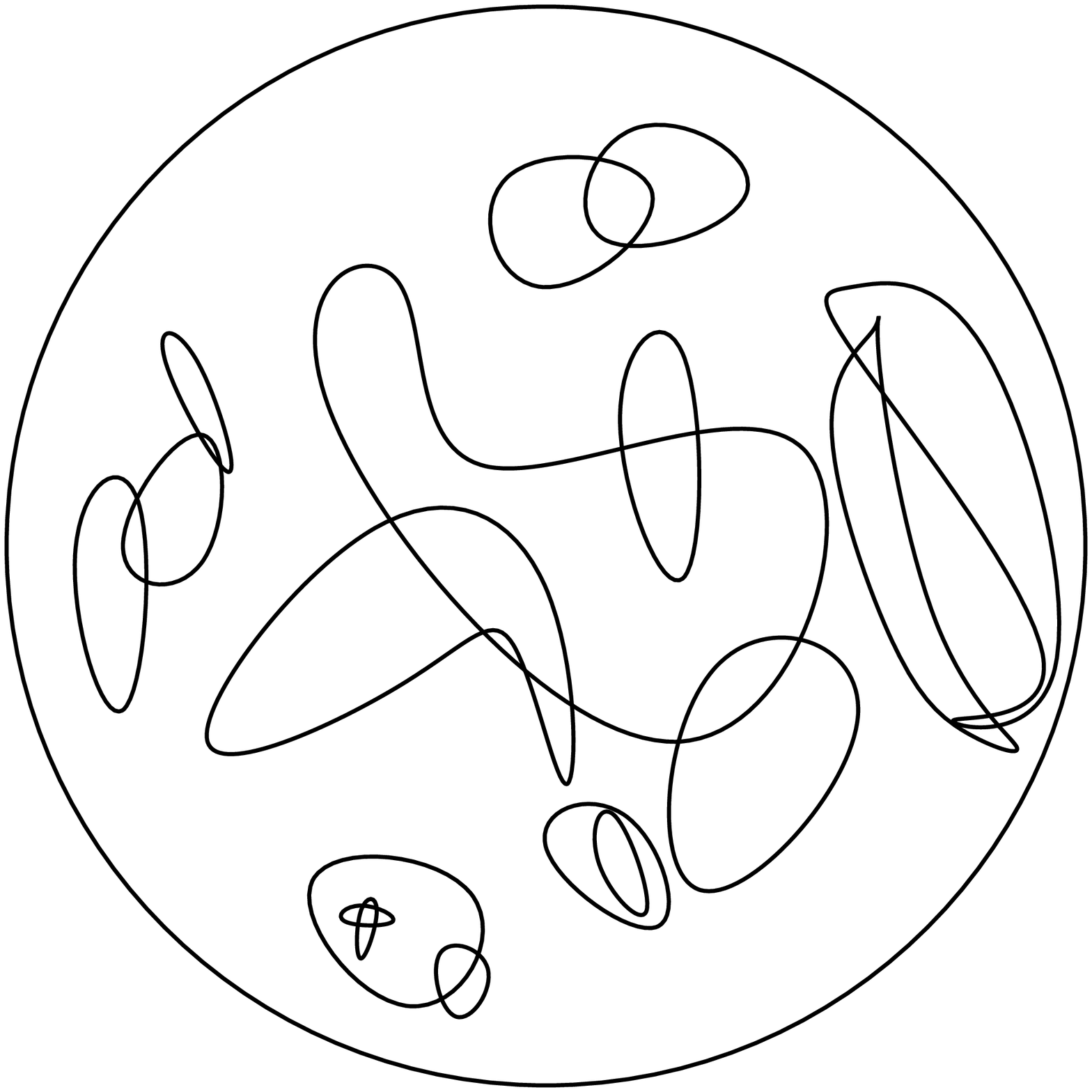}
\includegraphics [width=2in]{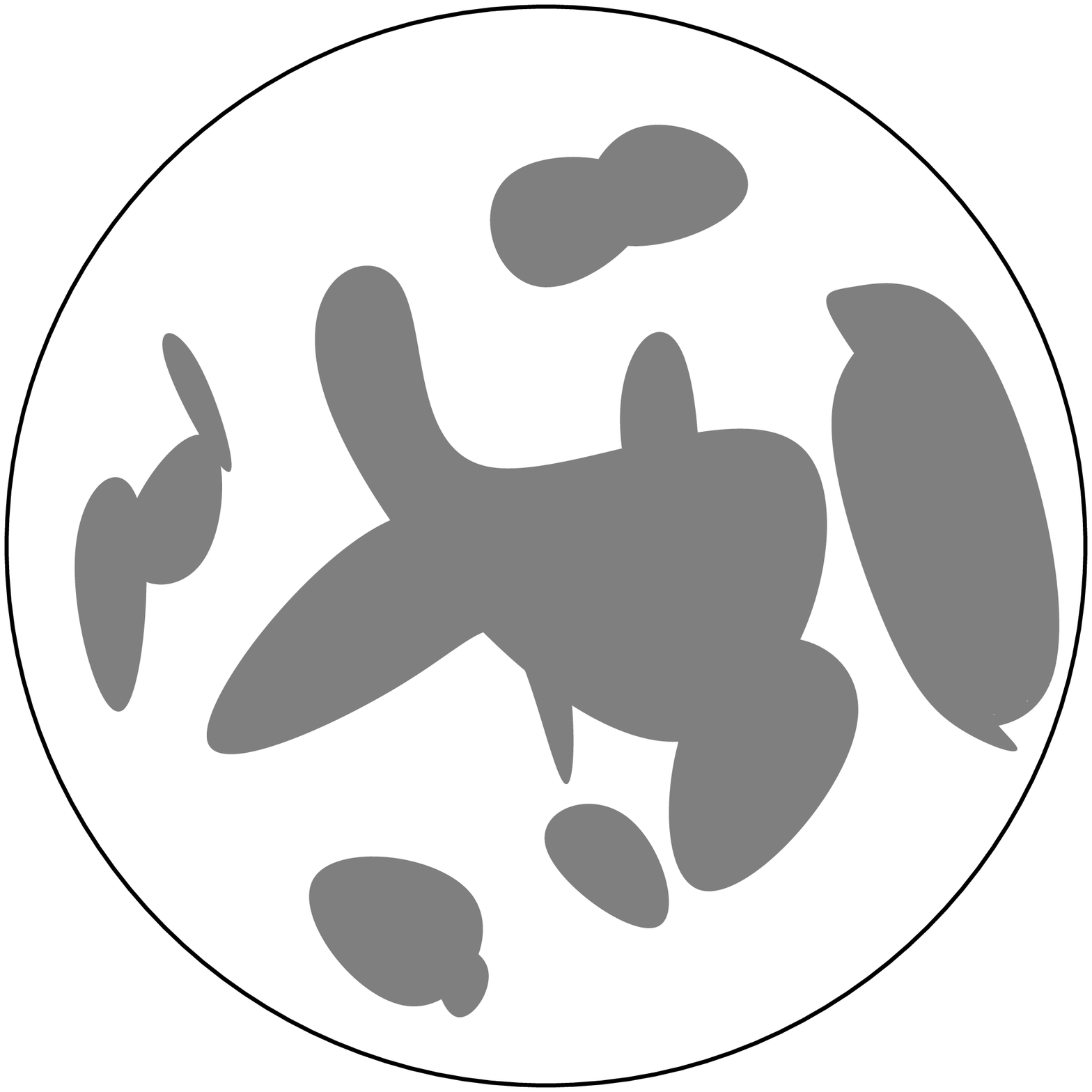}
\caption {A loup-soup and the fillings of its outermost clusters (sketch)}
\end {center}
\end {figure}

We write  $F(C)$ for the {\em filling} of $C$, i.e., for
the complement of the unbounded connected component of $\C \setminus \overline C$.
A cluster $C$ is called {\em outermost} if there exists no $C'$ such that $C \subset F(C')$.
The {\em outer boundary} of such an outermost cluster $C$ is the boundary of $F(C)$.
Denote by $\overline \Gamma$ the set of outer boundaries of outermost clusters of $\Gamma$.

\medbreak

Let us now state our main results:

\begin{theorem} \label{loopsoupsatisfiesaxioms}
Suppose that  $\Gamma$ is the Brownian loop-soup with intensity $c$ in $\U$.
\begin {itemize}
 \item
If $c \in (0,1]$,
then $\overline \Gamma$ is a random countable collection of disjoint simple loops that satisfies the conformal restriction axioms.
\item
If $c>1$, then there is almost surely only one cluster in $\mathcal C(\Gamma)$.
\end {itemize}
\end{theorem}

It follows immediately from this result
and \cite{sheffieldwerner1} that $\overline \Gamma$ is a CLE$_\kappa$ (according to the branching SLE$_{\kappa,
\kappa-6}$ based definition in \cite{Sh}, see also \cite {sheffieldwerner1} for alternative descriptions) for {\it some}
$\kappa \in (8/3,4]$.  We will derive the following correspondence:

\begin{theorem} \label{kappacorrespondence}
Fix $c \in (0,1]$ and let $\Gamma$ be a Brownian loop-soup of intensity $c$ on $\U$.
Then $\overline \Gamma$ is a CLE$_\kappa$ where $\kappa \in (8/3,4]$ is determined by the relation
$c=(3\kappa-8)(6-\kappa)/2\kappa$.
\end{theorem}

Since every $\kappa \in (8/3,4]$ is obtained for exactly one value of $c \in (0,1]$, this
implies the following:

\begin{corollary}
The random simple loop configurations satisfying the conformal restriction axioms
are precisely the CLE$_\kappa$ where $\kappa \in (8/3,4]$.
\end{corollary}

Recall from \cite {Be} that the Hausdorff dimension of an SLE$_\kappa$ curve is almost surely $1+ (\kappa /8)$. Our results therefore imply that the boundary of a loop-soup cluster of intensity $c \le 1$ has dimension
$$\frac { 37 -c - \sqrt { 25  + c^2 - 26 c }}{24} .$$
Note that just as for Mandelbrot's conjecture for the dimension of Brownian boundaries \cite {LSW4/3}, this statement does not involve SLE, but its proof does. In fact the
result about the dimension of Brownian boundaries can be viewed as the limit when $c \to 0$ of this one.

Furthermore we may define the {\em carpet}
of the CLE$_\kappa$ to be the random closed set $\U \setminus \cup_{\gamma \in \overline \Gamma} \,\interior \gamma$, and recall that SLE methods allowed \cite {MR2491617} to compute its
``expectation dimension'' in terms of $\kappa$.
The present loop-soup construction of CLE$_\kappa$ enables to prove
(see \cite {NW}) that this expectation dimension is indeed equal to its almost sure Hausdorff dimension $d$, and that in terms of $c$,
\begin {equation}
 d (c) =  \frac {187 - 7c + \sqrt {25 + c^2 - 26 c }}{96}
\end {equation}
The critical loop-soup (for $c=1$) corresponds therefore to a carpet of dimension $15/8$.

\medbreak

Another direct consequence of the previous results is the ``additivity property'' of CLE's:
If one considers two independent CLE's in the same simply connected domain $D$ with non-empty boundary, and looks at the union of these two, then either one can find a cluster whose boundary contains $\partial D$, or the outer boundaries of the obtained outermost clusters in this union form another CLE. This is simply due to the fact that each of the CLE's can be constructed via Brownian loop soups (of some
intensities $c_1$ and $c_2$) so that the union corresponds to a Brownian loop-soup of intensity $c_1+c_2$. This gives for instance a clean direct geometric meaning to the general idea (present on various occasions in the physics literature) that relates in some way two independent copies of the Ising model to the Gaussian Free Field in their large scale limit:
The outermost boundaries defined by the union of two independent CLE$_3$'s in a domain (recall \cite {ChSm} that CLE$_3$ is the scaling limit of the Ising model loops, and note that it corresponds to $c=1/2$) form a CLE$_4$ (which corresponds to level lines of the Gaussian Free Field, see \cite {SchSh,Dub} and to $c=1 = 1/2 + 1/2$).

\medbreak
This paper is structured as follows: We first study some properties of the Brownian loop-soup and of the clusters it defines. The main result of
Section \ref {S.2} is that when $c$ is not too large (i.e. is subcritical), the outer boundaries of outermost Brownian loop-soup clusters formally a random collection of disjoint
simple loops that does indeed satisfy the conformal restriction axioms.
By the main result of \cite {sheffieldwerner1}, this implies that they are CLE$_\kappa$ ensembles for some $\kappa$. In Section \ref {S.3}, we compare how loop-soups and SLE$_\kappa$ curves behave when one changes the domain that they are defined in, and we deduce from this the relation between $\kappa$ and $c$ in this subcritical phase.
In Section \ref {S.4}, we show that if the size of the clusters in a Brownian loop-soup satisfy a certain decay rate property, then the corresponding $c$ is necessary {\em strictly} subcritical. This enables to show that the loop-soup corresponding to $\kappa= 4$ is the only possible critical one, and completes the identification of all CLE$_\kappa$'s for $\kappa \in (8/3, 4]$ as loop-soup cluster boundaries.

\medbreak

The present paper is closely related to various other work on SLE and CLE, but
it will in fact directly use SLE results on only three distinct occasions: We use the main result of \cite {sheffieldwerner1} at the begining of Section \ref {S.3}, we use the standard SLE restriction properties from \cite {LSWr} and the description of CLE in terms of SLE excursions derived in \cite {sheffieldwerner1} in Section \ref {S.3}, and finally an estimate about the size of an SLE$_\kappa$ excursion borrowed from \cite {sheffieldwerner1}
 in Section \ref {S.4}.

\section{Loop-soup percolation}
\label {S.2}
The goal of this section is to prove the following proposition:

\begin{proposition} \label{p.satisfiesconformal} There exists a positive constant $c_0$ such that
 for all $c$ in $(0, c_0)$, the set $\overline \Gamma$ satisfies the conformal restriction axioms, whereas when $c$ is (strictly) greater than $c_0$, $\Gamma$ has only one cluster almost surely.
\end{proposition}

In the present section, we will not use any results from \cite {Sh,sheffieldwerner1}. The proposition will follow
immediately from a sequence of lemmas that we now state and prove. In these initial results, we will focus on properties of the collection $\overline \Gamm$ of (closures of) the clusters defined by $\Gamma$ (and we will not talk about outer boundaries or outermost loops).

It is easy to see (and we will justify this in a moment) that when $c$ is very large, there almost surely
exists just one cluster, and that this cluster is dense in $\U$, i.e., that
almost surely, $\overline \Gamm = \{ \overline \U \}$.

\begin {lemma} \label{l.confrest}
Suppose that $\mathbb P_c (\overline \Gamm = \{ \overline \U \}) < 1 $.
Let $U \subset \U$ denote some open subset of $\U$, and define $U^*$ to be the set obtained by removing from $U$ all the
(closures of) loop-soup clusters $\overline C$ that do not stay in $U$. Then, conditionally on
$U^*$ (with $U^* \not= \emptyset$), the set of loops of $\Gamma$ that do stay in $U^*$ is distributed like a Brownian loop-soup in $U^*$.
\end {lemma}
\begin {figure}[htbp]
\begin {center}
\includegraphics [width=2in]{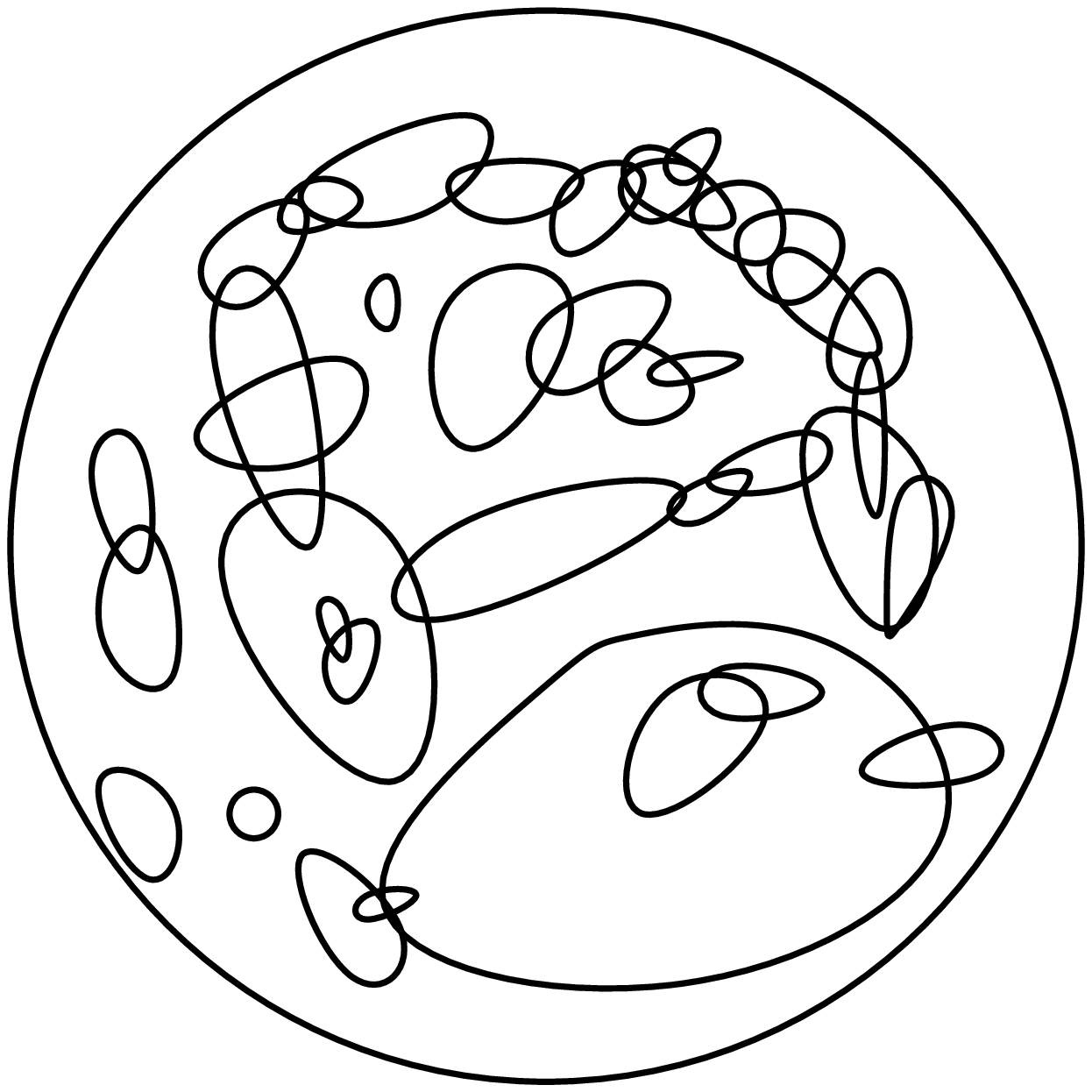}
\includegraphics [width=2in]{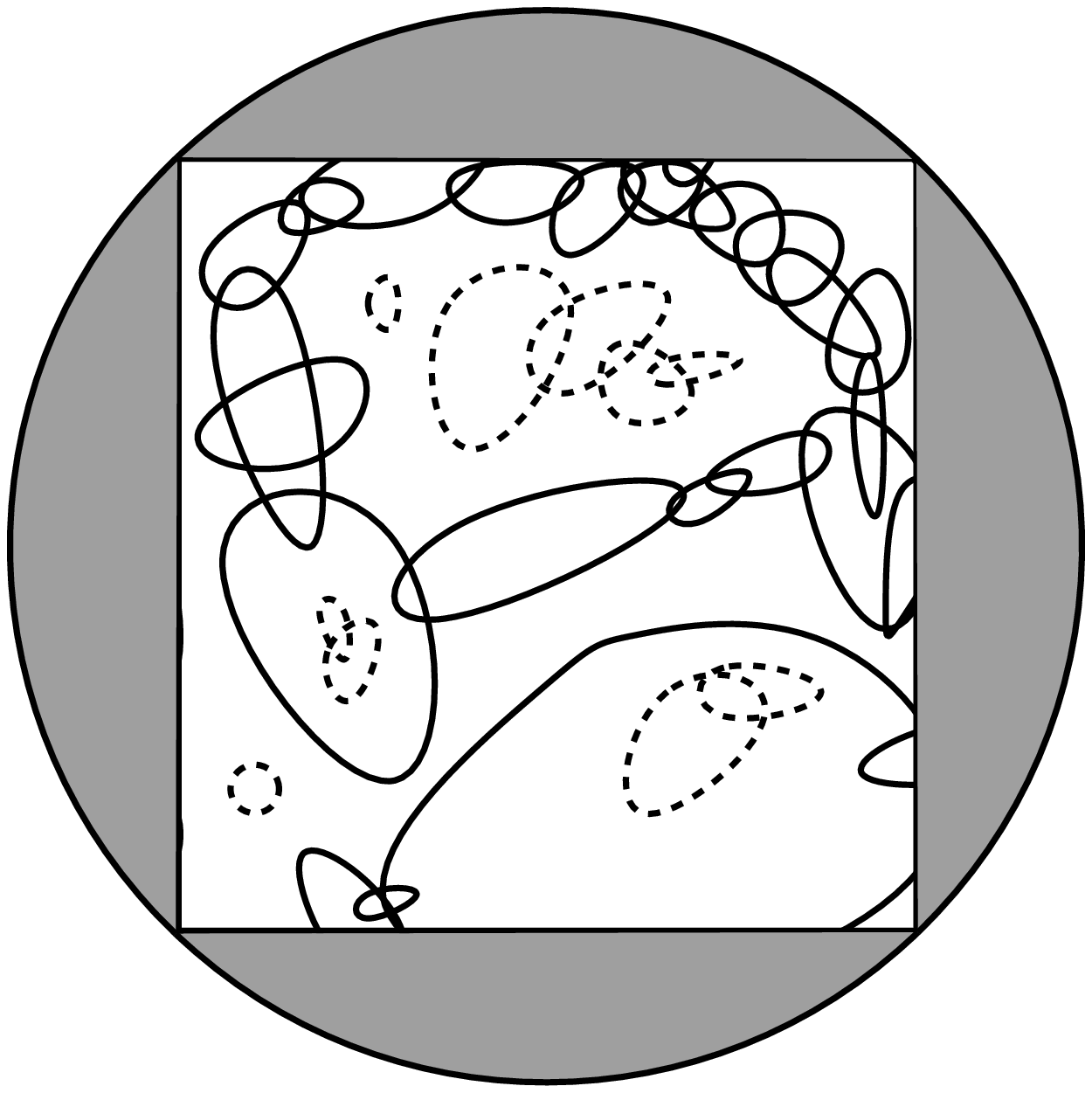}
\caption {Loop-soup clusters that stay in the rectangle $U$ are dashed (sketch)}
\end {center}
\end {figure}
Note that we have not yet proved at this point that in this case, $\overline \Gamma$ is a locally finite collection of disjoint simple loops.


\begin {proof}
Let us define for any $n \ge 1$, the set $U_n^*=U_n^* (U^*)$ to be the
 largest union of dyadic squares of side-lengths $2^{-n}$ that is contained in $U^*$ (note that this is a deterministic function of $U^*$).
For each $n \ge 1$, and for each union  $V_n$ of such dyadic squares the loop-soup restricted to $V_n$ is independent of the event $\{ U_n^* = V_n \}$.
It implies immediately that conditionally on $U^*$, the set of loops that do stay in $U_n^*$ is distributed like a Brownian loop-soup in $U_n^*$.
Since this holds for all $n$, the statement of the lemma follows.
\end {proof}

\begin{lemma} \label{l.biggerc-onecluster}
Suppose that $\mathbb P_c ( \overline \Gamm = \{ \overline \U \} ) < 1$ and that there is a $\mathbb P_c$ positive probability that $\overline \Gamm$ contains an element intersecting
the boundary of $\U$.  Then for all positive $c'$, $\mathbb P_{c+c'} ( \overline \Gamm = \{ \overline \U \} ) = 1$.
\end{lemma}

\begin{proof}
Assume the hypotheses of the lemma, and
let $A_1$ be the union of all elements of $\overline \Gamm$ that intersect some prescribed boundary arc $\partial$ of $\U$ of positive length.
By invariance under rotation, $\mathbb P_c ( A_1 \not= \emptyset ) > 0$.
Using the same
 argument as in the previous lemma, we get that if $U$ is a fixed open set such that $\overline U \subset \U$, then conditioned on the event $\overline U \cap \overline A_1 = \emptyset$,
the law of the set of loops in $\Gamma$ that are contained in $U$ is the same as its original law (since changing the set of loops within $U$ has no effect on $A_1$).
Since this holds for any $U$, conformal invariance of the loop soup implies that conditioned on $A_1$, the law
of the elements of $\Gamma$ that do not intersect $\overline A_1$ is that of independent copies of $\Gamma$ conformally mapped to
each component of $\U \setminus \overline A_1$.
Note that this in fact implies that the event that $A_1$ is empty is independent of $\Gamma$, and hence has probability zero or one (but we will not really need this).

The conformal radius $\rho_1$ of $\U \setminus \overline A_1$ seen from $0$ has a strictly positive probability to be smaller than one.
We now iterate the previous procedure: We let $U_2$ denote the connected component of $\U \setminus \overline A$ that contains the origin. Note that
the harmonic measure of $\partial_2:= \partial \cup A_1$ at $0$ in $U_1$ is clearly not smaller than the harmonic measure of $\partial$ in $\U$ at $0$ (a Brownian motion started at the origin that exits $\U$ in $\partial$ with necessarily exit $U_1$ through $\partial_2$). We now consider the loop-soup in this domain $U_2$, and we let $A_2$ denote the union of all loop-soup clusters that touch $\partial_2$. We then iterate the procedure, and note that the conformal radius of $U_n$ (from $0$) is dominated by a product of i.i.d. copies of $\rho_1$.

This shows  that for any positive $\delta$  one can almost surely find in $\overline \Gamma$ a finite sequence of clusters $\overline C_1, \ldots, \overline C_k$, such that
$d (C_j , C_{j+1}) = 0$ for all $j < k$,  such that $C_1$ touches $\partial$ and $d ( C_k, 0) \le \delta$.
By conformal invariance, it is easy to check that the same is true if we replace the origin by any fixed point $z$.
Hence, the statement holds almost surely, simultaneously for all
points $z$ with rational coordinates, for all rational $\delta$, and all boundary arcs of $\partial \U$ of positive length.

Note that almost surely, each loop of the loop soup surrounds some point with rational coordinates. We can therefore conclude (see Figure \ref {loops1000}) that almost surely,
any two clusters in $\overline \Gamm$ are ``connected'' via a finite sequence of adjacent clusters in $\overline \Gamm$.
\begin {figure}[htbp]
\begin {center}
\includegraphics [width=2.5in]{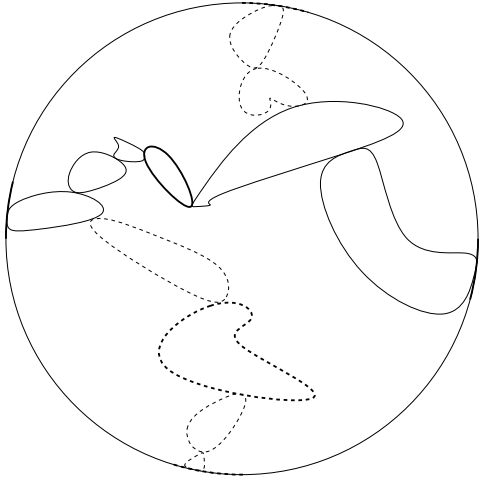}
\caption {The dark solid loop is part of a crossing of (light solid) loops from the left boundary segment to the right boundary segment.  The dark dotted loop is part of a crossing of loops from the lower boundary segment to the upper boundary segment.}
\label{loops1000}
\end {center}
\end {figure}

If we now augment $\Gamma$ by adding an independent Brownian loop-soup $\Gamma'$ of
intensity $c'$ for any given positive $c'$, the new loops will almost surely join together any two adjacent clusters of $\overline \Gamm$ into
a single cluster (this is just because any given point $z' \in \U$ -- for instance one chosen contact point between two adjacent clusters of $\overline \Gamm$ -- will almost surely be surrounded by infinitely many small loops of $\Gamma'$).
Furthermore, for an analogous reason, almost surely, any loop of $\Gamma'$ intersects some loop of $\Gamma$.
It follows that for all $c'>0$, $\mathbb P_{c+c'}$ almost surely, there exists just one single cluster, i.e., $\overline \Gamm = \{ \overline \U \}$.
\end{proof}

\begin{lemma} \label{l.firstpcdef}
There is a critical constant $c_0 \in [0, \infty]$ such that
\begin{enumerate}
 \item If  $c > c_0$, then $\mathbb P_c (\overline \Gamm = \{ \overline \U \}) = 1$.
 \item If  $c  \in (0, c_0)$, then $\mathbb P_c$ almost surely \begin{enumerate}
 \item[(a)] $\overline \Gamm$ has infinitely many elements.
 \item[(b)] No element of $\overline \Gamm$ intersects the boundary of $\U$.
 \item[(c)] No two elements of $\overline \Gamm$ intersect each other. \end{enumerate}
\end{enumerate}
\end{lemma}

\begin{proof}
Suppose that $\mathbb P_c ( \overline \Gamm = \{ \overline \U \} ) <1$;
if there is a $\mathbb P_c$ positive probability that two elements of $\overline \Gamm$ intersect each other, then (applying Lemma \ref {l.confrest} to some  $U$
that contains one but not the other with positive probability)  we find that there is a positive probability that an element of $\overline \Gamm$ intersects $\partial \U$.

Also, if $c>0$ and $\overline \Gamm$ has only finitely many elements with positive probability, then (with the same probability) at least one of these elements must intersect $\partial \U$ (since the loops of $\Gamma$ are dense in $\U$ a.s.).

Thus, Lemma \ref{l.biggerc-onecluster} implies that if $c_0$ is the supremum of $c$ for which (a), (b), and (c) hold almost surely, then $\Gamma$ has only one cluster $\mathbb P_c$ almost surely whenever $c>c_0$.
\end{proof}

We say $c$ is {\em subcritical} if the (a), (b), and (c) of Lemma \ref{l.firstpcdef} hold $\mathbb P_c$ almost surely. We will later show that $c_0$ is subcritical,
but we have not established that yet. We remark that the proof of Lemma \ref {l.firstpcdef} shows that in order to prove that some $c$ is subcritical, it suffices to check (b).

\begin{lemma}
\label{pc}
The $c_0$ of Lemma \ref{l.firstpcdef} lies in $(0,\infty)$.  Moreover, when $c>0$ is small enough, there are $\mathbb P_c$
almost surely continuous paths and loops in $D$ that
intersect no element of $\Gamma$.
\end{lemma}

\begin{proof}
We first prove the latter statement: for small $c>0$ there exist almost surely simple paths crossing $\U$ that intersect no element of $\Gamma$.  This
will also imply $c_0 >0$.  To this end we will couple the loop-soup with a well-known fractal percolation model.
The argument is similar in spirit to the one for multi-scale Poisson percolation in \cite{MR1409145}, Chapter 8.

Consider the unit square $D=(0,1)^2$ instead of $\U$. For each $n$, we will divide it into $4^n$ dyadic squares
of side-length $2^{-n}$. To each such square $C$, associate a Bernoulli random variable $X(C)$ equal to
one with probability $p$. We assume that the $X(C)$ are independent.
Then, define
\begin{equation} \label{e.mdef}
 M= [0,1]^2 \setminus \bigcup_{C \ : \ X(C)=0} C .
\end{equation}
This is the fractal percolation model introduced by Mandelbrot in \cite{MR665254} (see also the book \cite{MR1409145}).
It is very easy to see that the area of $M$ is almost surely zero as soon as $p<1$.
Chayes, Chayes and Durrett \cite{MR931500}
have shown that this model exhibits a phase transition:
There exists a $p_c$, such that for all $p \ge p_c$, $M$ connects the left and right sides of the
unit square with positive probability, whereas for all $p < p_c$, this is a.s.\ not the case
(note that in fact, if $p \le 1/4$, then $M$ is almost surely empty by a standard martingale argument).
Here, we will only use the fact that for $p$ large enough (but less than one), $M$ connects the two opposite sides of the
unit square with positive probability.  We remark that the proof in \cite{MR931500} actually gives (for large $p$) a positive probability that there
exists a continuous path from the left to right side of the unit square in $M$ that can be parameterized as $t \to (x(t),y(t))$ where $t \in [0,1]$
and $x(t)$ is non-decreasing. It also shows (modulo a straightforward FKG-type argument) that $M$ contains loops with positive probability.

Now, let us consider a loop-soup with intensity $c$ in the unit square.
For each loop $l$, let $d(l) \in (0,1)$ denote its $L^1$-diameter (i.e., the maximal variation of the $x$-coordinate or of the $y$-coordinate), and define $n(l) \ge 0$
 in such a way that $d(l) \in [2^{-n-1}, 2^{-n})$.
Note that $l$ can intersect at most 4 different dyadic squares with side-length $2^{-n}$. We
can therefore associate
in a deterministic (scale-invariant and translation-invariant)
manner to each loop $l \subset (0,1)^2$, a dyadic site $s = (j2^{-n(l)}, j'2^{-n(l)}) \in (0,1)^2$  such that
$l$ is contained in the square $S_l$ with side-length $2 \times 2^{-n(l)}$ and bottom-left corner at $s$. Note that $S_l \subset (0,2)^2$.

We are first going to replace all loops $l$ in the loop-soup by the squares $S_l$.
This clearly enlarges the obtained clusters. By the scale-invariance and the Poissonian character of
the loop-soup, for each fixed square
\begin{equation}
\label{e.stype} S=(j2^{-m}, (j+2) 2^{-m}) \times (j' 2^{-m}, (j'+2) 2^{-m}) \subset (0,2)^2,
\end{equation}
the probability that there exists no loop $l$ in the loop-soup such that
$S_l=S$ is (at least) equal to $\exp ( - bc)$ for some positive constant $b$ (that is independent of $m$).
Furthermore, all these events (when one lets $S$ vary) are independent.

Hence, we see that the loop-soup percolation process is dominated by a variant of Mandelbrot's fractal percolation model:
let $\tilde X$ denote an independent percolation on squares of type \eqref{e.stype},
with each $\tilde X(S)$ equal to $1$ with probability $\tilde p = \exp (-bc)$ and define the random compact set
\begin{equation}
\tilde M= [0,2]^2 \setminus \bigcup_{S \ : \ \tilde X(S)=0} S.
\end{equation}
Note that in the coupling between the loop-soup and $\tilde M$ described above, the distance between $\tilde M$ and each fixed loop in the loop-soup
is (strictly) positive almost surely.  In particular, $\tilde M$ is contained in the complement of the union of the all the loops (and their interiors).

We now claim that this variant of the percolation model is dominated by Mandelbrot's original percolation model with a larger (but still less than 1) value of $p = p(\tilde p)$ that we will choose in a moment.
To see this, let $\hat X$ be a $\hat p$-percolation (with $\hat p = p^{1/4}$) on the set of $(C,S)$ pairs with $C$ a dyadic square and $S$ a square comprised of $C$ and three
of its neighbors.  Take
$$X(C) = \min_{S\ : \ C \subset S} \hat X(C,S) \hbox { and } \tilde X(S) = \max_{C \ : \ C \subset S} \hat X(C,S).$$
Clearly $X(C)$ is a Bernoulli percolation with parameter $p = \hat p^4$ and $\tilde X(S)$ is a Bernoulli percolation with parameter $1-(1-\hat p)^4$.
Let us now choose $p$ in such a way that $\tilde p  = 1 - (1- \hat p)^4$. Note that $p$ tends monotonically to $1$ as $\tilde p$ tends to $1$.
Hence, by taking $\tilde p$ sufficiently close to $1$, i.e. $c$ sufficiently small, we can ensure that $p$ is as close to $1$ as we want (so that $M$
contains paths and loops with positive probability). But
in our coupling, by construction we have
$$X(C) \leq \min_{C \subset S} \tilde X(S),$$
 and thus $M \subset \tilde M$.

We have now shown that $c_0 > 0$, but we still have to show that $c_0 < \infty$.  We use a similar coupling with the fractal percolation
model. For any dyadic square that does not touch the boundary of $[0,1]^2$, we let $X(C)$ be $0$ if $C$ is surrounded by a loop in $\Gamma$ that is contained in the set of eight neighboring dyadic squares to $C$ (of the same size).
The $X(C)$ are i.i.d.\ (for all $C$ whose eight neighbors are contained in $D$), and have a small probability (say smaller than $1/4$) of being $1$ when $c$ is taken sufficiently large.
We now use the fact, mentioned above, that if $p \le 1/4$, then $M$ is almost surely empty; from this we conclude easily that almost surely, for each $C$ (whose incident neighbors
are in $D$) every point in $C$ is surrounded by a loop in $\Gamma$ almost surely.  It follows immediately that almost surely, every point in $D$ is surrounded by a loop in $\Gamma$.
This implies that almost surely all loops of $\Gamma$ belong to the same cluster (since otherwise there would be a point on the boundary of a cluster that was not surrounded by a loop).
\end{proof}

\begin{lemma} \label{exponentialcrossingdecay}
If $c$ is sub-critical, the probability that there are $k$
disjoint chains of loops in $\overline \Gamm$ crossing from the inside to the outside of a fixed annulus decays exponentially in $k$.
\end{lemma}

\begin{proof}
We know that $c$ is subcritical, so that for all $r<1$, the probability that there exists a crossing of the annulus $\{ z \ : \ r < | z| < 1 -1 /n \}$
by a chain of loops goes to $0$ as $n\to \infty$. Hence, there exists $r_1 \in (r, 1)$ such that the probability that there is a single crossing of $\{ z \ : \ r < | z | < r_1 \}$ is strictly smaller than one.
Hence, it follows easily that if we consider any given annulus, $\{ z \ : \  r < | z - z_0 | < r' \} \subset \U$, there exists a positive probability that no cluster of the loop-soup crosses the annulus (just consider the two independent events of positive probability that the loop-soup restricted to $\{ z \ : \ |z -z_0 | < r'/r_1\}$ contains no crossing of the annulus, and that no loop intersects both the circles of radius $r'$ and $r'/r_1$ around $z_0$).  In other words, the probability that there exists a crossing of the annulus is strictly smaller than one.
 The result then follows from the BK inequality for Poisson point processes (see for instance \cite {vdB} and the references therein).
\end{proof}

As a consequence (letting $k \to \infty$), we see that for each fixed annulus, the probability that it is crossed by infinitely many disjoint chains of loops is zero.
From this we may deduce the following:

\begin{lemma} \label{l.simpleloops}
If $c$ is sub-critical, the set $\overline \Gamm$ is almost surely locally finite, and moreover the elements of $\overline \Gamma$ are almost surely continuous simple loops.
\end{lemma}
\begin{proof}
If the $\overline \Gamm$ are not locally finite, then for some positive $\epsilon$ there exists a point $z \in \overline \U$ and a sequence $\overline C_n$ of
elements of $\overline \Gamm$ of size greater than $\eps$ such that $d (z, C_n) \to 0$ as $n \to \infty$.
 Hence any annulus with outer radius less than $\epsilon/2$ that surrounds $z$ will have infinitely many crossings.  By Lemma \ref{exponentialcrossingdecay},
the probability that such a point exists is zero (this result just follows from the lemma by considering a countable collection of annuli such that each point in $\U$ is surrounded by arbitrarily small annuli from this collection).

No element of $\overline \Gamm$  can have a cut point (since almost surely no two loops of $\Gamma$ intersect at only a single point -- recall also that the elements of $\Gamma$ are all {\em simple} loops),
so if the boundary is a continuous curve, it must be simple.  If the boundary has a point of discontinuity, then there must be infinitely many chains of loops in $\Gamma$ crossing an annulus surrounding that point, which Lemma \ref{exponentialcrossingdecay} again rules out.
\end{proof}

Proposition \ref{p.satisfiesconformal} now follows from the results proved in this section: Lemma \ref{pc} gives the existence of $c_0$, Lemma \ref{l.simpleloops} implies that the loops are simple and locally finite a.s., and Lemma \ref{l.confrest} yields the restriction property.

\section {Relation between $c$ and $\kappa$}
\label {S.3}

The main result of \cite{sheffieldwerner1} combined with Proposition \ref {p.satisfiesconformal}
now implies the following:

\begin{corollary} \label{CLEcor}
If $c$ is sub-critical, then the set $\overline \Gamma$ is a CLE$_\kappa$ for some $\kappa \in (8/3, 4]$.
In other words, for such $c$, the $\overline \Gamma$
is equivalent in law to the loop ensemble constructed in \cite{Sh} via branching SLE$_{\kappa, \kappa-6}$.
\end{corollary}

We will now identify $\kappa$ in terms of $c$.
\begin{proposition} \label{p.ckappa} For all subcritical $c$, the set $\overline \Gamma$ is in fact a CLE$_\kappa$ with $c=(3\kappa-8)(6-\kappa)/2\kappa$.
\end{proposition}

Note that in our proof, we will not really use the description of CLE$_\kappa$ via branching SLE$_{\kappa, \kappa -6}$. We will only use the description of the conformal loop-ensemble via its pinned loop measure, as described in \cite {sheffieldwerner1}, which we now briefly recall.

Suppose that $\overline \Gamma$ is a random loop ensemble that satisfies the conformal restriction axioms. Consider its version in the upper half-plane $\H$, and consider the loop
$\gamma (i)$ of $\overline \Gamma$ that surrounds $i$ (it is very easy to see that this loop almost surely exists, see \cite {sheffieldwerner1}). Let us now consider the law of $\gamma(i)$ conditioned on the event $\{d( 0, \gamma (i)) \le \eps \}$. It is shown in \cite {sheffieldwerner1} that this law converges when $\eps \to 0$ to some limit
$P^i$, and furthermore that for some $\kappa \in (8/3, 4]$, $P^i$ is equal to an SLE-excursion law $P^{i, \kappa}$ that we will describe in the next paragraph. Furthermore, when this is the case, it turns out that the entire family $\overline \Gamma$ is a CLE$_\kappa$ for this value of $\kappa$.

Consider an SLE$_\kappa$ in the upper half-plane, started from the point $\eps >0$ on the real axis to the point $0$ (on the real axis as well). Such an SLE path will typically be very small when $\eps$ is small. However, one can show that the limit when $\eps \to 0$ of (the law of) this SLE, conditioned on disconnecting $i$ from $\infty$ in the upper half-plane exists. This limit (i.e., its law) is what we call $P^{i, \kappa}$.

Conformal invariance of SLE$_\kappa$ enables to define an analogous measure in other simply connected domains. Suppose for instance that
$H = \{ z \in \H \ : \ | z| < 3 \}$. We can again consider the limit of the law of SLE$_\kappa$ in $H$ from $\eps$ to $0$, and conditioned to disconnect
$i$ from $3i$. This limit is a probability measure $P_H^{i, \kappa}$ that can also be viewed as the image of $P^{i , \kappa}$ under the conformal map
$\Phi$ from $\H$ onto $H$ that keeps the points $0$ and $i$ invariant.  Note that the same argument holds for other choices of $H$, but choosing this particular one will be enough $H$ to identify the relationship between $c$ and $\kappa$.

One can use SLE techniques to compare ``directly'' the laws of an SLE $\gamma$ from $\eps$ to $0$ in $\H$ and of an SLE $\gamma'$ also from $\eps$ to $0$ in $H$.
More precisely, the SLE martingale derived in \cite {LSWr} show that the Radon-Nikodym of
the former with respect to the latter is a multiple of
$$ \exp  ( - c L ( \gamma, \H \setminus H ; \H ))$$
on the set of loops $\gamma$ that stay in $H$ (recall that $L(A,A';D)$ denotes the $\mu$-mass of the set of Brownian loops in $D$ that intersect both $A$ and $A'$),
where
$$  c = c (\kappa) =\frac {(3\kappa-8)(6-\kappa)}{2\kappa}.$$
This absolute continuity relation is valid for all $\eps$, and it therefore follows that it still holds after passing to the previous limit $\eps \to 0$, i.e., for the
two probability measures $P^{i, \kappa}$ and $P^{i, \kappa}_H$. This can be viewed as a property of $P^{i, \kappa}$ itself because $P^{i, \kappa}_H$ is the conformal image of $P^{i, \kappa}$ under $\Phi$.

Note that the function $\kappa \mapsto c(\kappa)$ is strictly increasing on the interval $(8/3, 4]$. Only one value of $\kappa$ corresponds to each value of $c \in (0,1]$.
Hence, in order to identify the value $\kappa$ associated to a family $\overline \Gamma$ of loops satisfying the conformal restriction axioms, it suffices to check that the probability measure $P^i$ satisfies the corresponding absolute continuity relation for the corresponding value of $c$.

\begin{proof}
Suppose that $c$ is subcritical, and that $\overline \Gamma$ is the corresponding family (of outer boundaries of outermost clusters).
By restricting ourselves to the loops in the loop-soup that stay in $H$, we can define on the same probability space, a sample of the loop-soup in $H$.
Let $\gamma(i)$ denote the loop in $\overline \Gamma$ that surrounds $i$, and let $\gamma'(i)$ denote the loop in $\overline \Gamma'$ that surrounds $i$.
Let us start with sampling the loop-soup in $H'$ (which defines $\gamma'(i)$). Then, on the top of that, in order to obtain the loop-soup in $\H$, we only add loops that intersect $\H \setminus H$. Note that if $\gamma(i) \not= \gamma'(i)$, then $\gamma(i) \notin H$. Furthermore, in order for
$\gamma (i) = \gamma'(i)$ it suffices that:
\begin {itemize}
 \item One did not add any loop in the loop-soup that intersects $\gamma'(i)$ (note that this probability -- conditionally on $\gamma'(i)$ --  is
$\exp ( - c L ( \gamma'(i), \H \setminus H ; \H )$).
\item One did not create another disjoint cluster that goes ``around'' $\gamma'(i)$. When $\gamma'(i)$ already intersects the disk of radius $\eps$, the conditional probability that one creates such an additional cluster goes to $0$ as $\eps \to 0$ (using the BK inequality for Poisson point processes, as in the proof of Lemma \ref{exponentialcrossingdecay}.
\end {itemize}
If we now condition on the event that $\gamma'(i)$ intersects the disk of radius $\eps$ and let $\eps \to 0$, it follows that under the limiting law $P^i$ satisfies the absolute continuity relation that we are after: On the set of curves $\gamma$ that stay in $H$, the Radon-Nikodym derivative of $P^i$ with respect to the measure defined directly  in $H$ instead of $\H$ is $\exp ( - c L ( \gamma, \H \setminus H ; \H )$).
Hence, it can only be $P^{ i, \kappa}$ for $c=c(\kappa)$.
\end {proof}

This identification allows us to give a short proof of the following fact, which will be instrumental in the next section:

\begin {proposition}
 \label {p.cosubcritical}
$c_0$ is subcritical.
\end {proposition}

Note that the standard arguments developed in the context of Mandelbrot percolation
(see \cite{MR1409145,MR931500}) can be easily adapted to the present setting in order to prove that $c_0$ is subcritical, but not in the sense we have defined it. It shows for instance easily that at $c_0$, there exist paths and loops that intersect no loop in the loop-soup, but non-trivial additional work would then be required in order to deduce that no cluster touches the boundary of the domain. Since we have this identification via SLE$_\kappa$ loops at our disposal, it is natural to prove this result in the following way:

\begin {proof}
Let $C_c$ be the outermost cluster surrounding $i$ if the loop-soup (in $\H$) has intensity $c$.
If we take the usual coupled Poisson point of view in which the set $\Gamma = \Gamma(c)$ is increasing in $c$ (with loops ``appearing'' at random times up to time $c$) then we have by definition that almost surely $C_{c_0} = \cup_{c < c_0} C_c$ (this is simply because, almost surely, no loop appears exactly at time $c_0$).  Let $d(c)$ denote the
Euclidean distance between $\overline C_c$ and the segment $[1, 2]$.
Clearly $d(c) > 0 $ almost surely for each $c < c_0$ and $d(c_0) = \lim_{c \to c_0-} d(c)$.

By the remark after Lemma \ref{l.firstpcdef}, we know that in order to prove that $c_0$ is subcritical, it suffices to show that $d(c_0) > 0$ almost surely (by Moebius invariance, this will imply that almost surely, no cluster touches the boundary of $\H$).
Note that $d$ is a non-increasing function of the loop-soup configuration (i.e., adding more loops to a configuration can not increase the corresponding distance $d$).
Similarly, the event $E_\eps$ that the outermost cluster surrounding $i$ does intersect the $\eps$-neighborhood of the origin is an increasing event (i.e., adding more loops to a configuration can only help this event to hold).  Hence it follows that for each $\eps$, the random variable $d(c)$ is negatively correlated with the event $E_\eps$.
Letting $\eps \to 0$, we get that (for subcritical $c$),
 the law of $d(c)$ is ``bounded from below'' by the law of the distance between the curve $\gamma$ (defined under the probability measure $P^{i, \kappa}$) and $[1,2]$, in the sense that for any positive $x$,
$$ P ( d(c) \ge  x ) \ \ge \  P^{i, \kappa} ( d ( \gamma, [1,2]) \ge x ).$$
But we also know that $c_0 \le 1$, so that $\kappa_0 := \lim_{c \to c_0-} \kappa (c) \le 4$.
It follows readily that for all $c < c_0$ and for all $x$,
$$ \lim_{c \to c_0-} P ( d(c) \ge  x ) \ \ge\  \lim_{\kappa \to \kappa_0-} P^{i, \kappa} ( d( \gamma, [1,2]) \ge  x )
\ \ge\  P^{i, \kappa_0} ( d(\gamma, [1,2] ) \ge 2x).$$
But we know that for any $\kappa \le 4$, the SLE excursion $\gamma$ stays almost surely away from $[1,2]$.
Putting the pieces together, we get indeed that
\begin {eqnarray*}
P ( d (c_0)  > 0 )  &= & \lim_{x \to 0+} P ( d(c_0) \ge  x )
\ \ge \  \lim_{x \to 0+} \lim_{c \to c_0-} P ( d (c) \ge x )  \\
&\ge&  \lim_{x \to 0+} P^{i, \kappa_0} ( d ( \gamma, [1,2] ) \ge 2x)
\ \ge \  P^{i, \kappa_0} ( d( \gamma, [1,2] ) > 0 )
 \  =   \  1 .
\end {eqnarray*}
\end {proof}

\section{Identifying the critical intensity $c_0$}
\label {S.4}

\subsection {Statement and comments}

The statements of our main results, Theorems \ref{loopsoupsatisfiesaxioms} and
\ref{kappacorrespondence}, are mostly contained in the results of the previous section:
Corollary \ref{CLEcor}, Proposition \ref{p.ckappa}, Proposition \ref{p.cosubcritical}.  It remains only to prove the following statement:

\begin{proposition} \label{criticaltheorem}
The critical constant of Lemma \ref{l.firstpcdef} is $c_0 = 1$ (which corresponds to $\kappa = 4$, by Proposition \ref{p.ckappa}).
\end{proposition}

Propositions \ref{p.satisfiesconformal}, \ref {p.ckappa} and the main result of \cite{sheffieldwerner1} already imply that we cannot have $c_0 > 1$,
since in this case the $\overline \Gamma$  corresponding to $c \in (1, c_0)$ would give additional random loop collections satisfying the conformal axioms (beyond the CLE$_\kappa$ with $\kappa \in (8/3,4]$), which was ruled out in \cite{sheffieldwerner1}.  It remains only to rule out the possibility that $c_0 <1$.

 The proof of this fact is not straightforward, and it requires some new notation and several additional lemmas. Let us first outline the very rough idea. Suppose that $c_0 < 1$. This means that at $c_0$, the loop-soup cluster boundaries are described with $SLE_\kappa$-type loops for $\kappa = \kappa ( c_0) < 4$. Certain estimates on SLE show that SLE curves for $\kappa < 4$ have a probability to get $\epsilon$ close to a boundary arc of the domain that decays quickly in $\epsilon$, and that this fails to hold for SLE$_4$.  In fact, we will use a very closely related result from \cite{sheffieldwerner1} that
roughly shows that the probability that the diameter of the set of loops intersecting a small ball on the boundary of $\H$ is large,  decays quickly with the size of this ball when $\kappa < 4$.
  Hence, one can intuitively guess that when $\kappa < 4$, two big clusters will be unlikely to be very close, i.e., in some sense, there is ``some space'' between the clusters. Therefore, adding a loop-soup with a very small intensity $c'$
on top of the loop-soup with intensity $c_0$ might not be sufficient to make all clusters percolate, and this would contradict the fact that $c_0$ is critical.

We find it convenient in this section to work with a loop-soup in the upper half plane $\H$ instead of the unit disk.
To show that there are distinct clusters in such a union of two CLEs $\overline \Gamma$ and $\overline \Gamma'$,
we will start with the semi-disk $A_1$ of radius $1$ centered at the origin.
We then add all the loops in $\overline \Gamma$ that hit $A_1$, add the loops in ${\overline \Gamma}'$ that hit
those loops, add the loops in $\overline \Gamma$ that hit those loops, etc. and try to show that in some sense the size of this
growing set remains bounded almost surely. The key to the proof is to find the right way to describe this ``size'', as the usual quantities
such as harmonic measure or capacity turn out not to be well-suited here.

\subsection {An intermediate way to measure the size of sets}

We will now define a generalization of the usual half-plane capacity.
Suppose that $\alpha \in (0,1]$, and that $A$ is a bounded closed subset of the upper half-plane $\H$. We define
\begin{equation} \label{hcappsi} M(A) = M_\alpha (A) := \lim_{s \to \infty} s \E ( (\Im B^{is}_{\tau(A)})^\alpha),\end{equation}
where $B^{is}$ is a Brownian motion started at $is$ stopped
at the first time $\tau(A)$ that it hits $A \cup \R$.  Note that $M_1= \hcap$ is just the
usual half-plane capacity used in the context of chordal Loewner chains, whereas $\lim_{\alpha \to 0+} M_\alpha (A)$ is the harmonic
measure of $A \cap \HH$ ``viewed from infinity''.
Recall that standard properties of planar Brownian motion imply that
the limit in (\ref {hcappsi})
 necessarily exists, that it is finite, and that for some universal constant $C_0$ and for any $r$ such that $A$ is a subset of the disk of radius $r$ centered at the origin, the limit is equal to
$C_0 r^{-1} \times \E ( (\Im B_{\tau(A)})^\alpha )$
where the Brownian motion $B$ starts at a random point $r e^{i \theta}$, where $\theta$ distributed according to the density $\sin (\theta)  d \theta / 2$ on $[0, \pi]$.

A {\em hull} is defined to be a bounded closed subset of $\H$ whose complement in $\H$ is simply connected.
The union of two hulls $A$ and $A'$ is not necessarily a hull, but we denote by $A \underline \cup A'$ the hull whose
complement is the unbounded component of $\H \setminus (A \cup A')$.

When $A$ is a hull, let us denote by $\Phi_A: \H \setminus A \to \H$ the conformal map normalized at infinity by  $\lim_{z \to \infty} \Phi_A(z) - z = 0$.
 Recall that for all $z \in \HH \setminus A$, $\Im ( \Phi_A (z) ) \le \Im ( z)$.
Then, when $A'$ is another hull, the set $\Phi_A ( A' \cap ( \H \setminus A))$ is not necessarily a hull. But we can still define the unbounded connected component of its complement in the upper half-plane, and take its complement. We call it $\Phi_{A} (A')$ (by a slight abuse of notation).

It is well-known and follows immediately from the definition of half-plane capacity that for any bounded closed $A$ and any positive $a$,
$\hcap (aA) = a^2 \hcap(A)$. Similarly, for any two $A$ and $A'$, $\hcap(A \cup A') \leq \hcap(A) +\hcap(A')$.
Furthermore $\hcap$ is increasing with respect to $A$, and behaves additively with repect to composition of conformal maps for hulls.

We will now collect easy generalizations of some of these four facts.
 Observe first that
for any positive $a$ we have
\begin{equation}\label{Mscaling} M(a A) = a^{\alpha+1} M(A). \end{equation}
Similarly, we have that for any two (bounded closed) $A$ and $A'$,
\begin {equation}
 \label{Munionsubadditive}
M(A  \cup A') \leq M(A) + M(A').
\end {equation}
This follows from the definition \eqref{hcappsi} and the fact that for each sample of the Brownian motion we have
$$(\Im B_{\tau(A  \cup A')})^\alpha = (\Im B_{\tau(A) \wedge \tau(A')})^\alpha \leq  (\Im B_{\tau(A)})^\alpha +  (\Im B_{\tau(A')})^\alpha.$$
Applying the optional stopping time theorem to the local supermartingale $(\Im B_t)^\alpha$, we know that
\begin{equation} \label{Msubset}
A \subset A' \hbox{ implies }  M(A) \leq M(A'),
\end{equation}
since $$\E\bigl((\Im B_{\tau(A) })^\alpha | B_{\tau(A')} \bigr) \leq (\Im B_{\tau(A')})^\alpha.$$
We next claim that
for any three given hulls $A'$, $A_1$ and $A_2$, we have
\begin {equation} \label{Mtriple}
 M(\Phi_{A_1 \underline \cup A_2} (A')) \le M (\Phi_{A_1} (A')).
\end {equation}
To verify the claim, note first that for $A_3= \Phi_{A_1} (A_2)$, we have $\Phi_{A_3} \circ \Phi_{A_1} = \Phi_{A_1 \underline \cup A_2}$.
Recall that $\Im \Phi_{A_3} (z) \leq \Im(z)$ for all $z \in \H$, so that in particular, $\Im (\Phi_{A_3} (B_{\sigma})) \leq \Im (B_{\sigma})$ where
$\sigma$ is the first hitting time of $\Phi_{A_1} (A')$ by the Brownian motion.  We let the starting point of the Brownian motion tend to infinity as before, and the claim follows.

\medbreak

It will be useful to compare $M(A)$ with some quantities involving dyadic squares and rectangles that $A$ intersects.  (This is similar in spirit to the estimates for half-plane capacity in terms of hyperbolic geometry given in \cite{LLN}.)  We will consider the usual hyperbolic tiling of $\H$ by squares of the form $[a 2^j, (a+1) 2^j] \times [2^j, 2^{j+1}]$, for integers $a,j$.
Let $\mathcal S$ be the set of all such squares.
For each hull $A$, we define $\mathcal S (A)$ to be the set of squares in $\mathcal S$ that $A$ intersects, and we let $\hat A$ be the union of these squares, i.e.
$$ \hat A = \cup_{S \in \mathcal S (A)} S.$$
\begin {figure}[htbp]
\begin {center}
\includegraphics [width=5in]{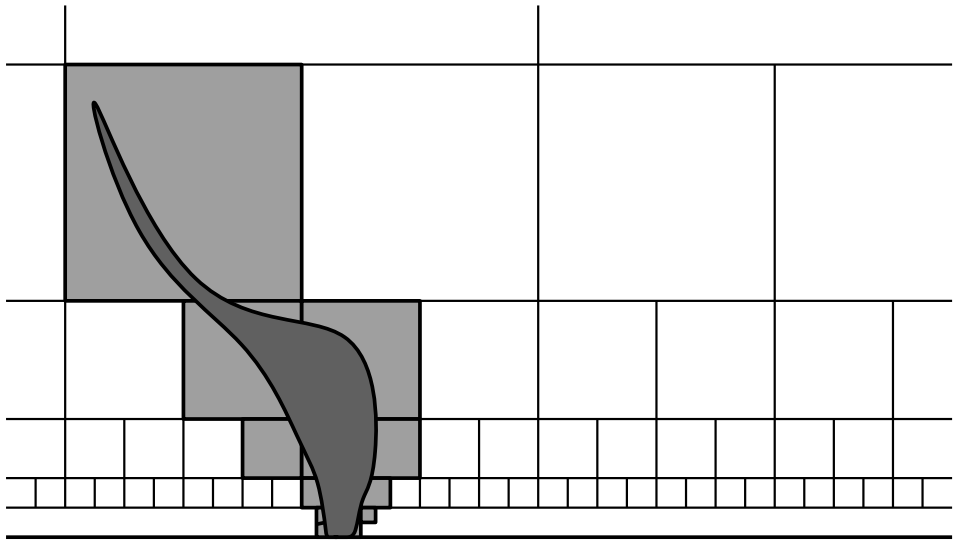}
\caption {A hull $A$ and its corresponding $\hat A$}
\end {center}
\end {figure}

\begin{lemma}\label{MhatMagree}
There exists a universal positive constant $C$ such that for any hull $A$,
$$ C M ( \hat A) \le M(A) \le M( \hat A).$$
\end{lemma}

\begin{proof}
Clearly, $M (\hat A) \ge M (A)$ by \eqref{Msubset} since $A \subset \hat A$. On the other hand,
if we stop a Brownian motion at the first time it hits $\hat A$ i.e. a square $S$ of
$\mathcal S(A)$, then it has a bounded probability of later hitting $A$ at a point of about the same height, up to constant factor: This can
be seen, for example, by bounding below the probability that (after this hitting time of $\hat A$) the Brownian motion makes a loop around $S$ before it hits any square
of $\mathcal S$ that
is not adjacent to $S$, which would in particular imply that it hits $A$ during that time. This probability is universal, and the lemma follows.
\end{proof}

\begin{lemma}\label{mhatsum}
There exists a universal positive constant $C'$ such that for any hull $A$,
$$
C' \sum_{S \in {\cal S}(A)} M (S)
\le
M (A)
\le
\sum_{S \in {\cal S}(A)} M (S) .
$$
\end{lemma}

\begin{proof} The right-hand inequality is obvious by (\ref{Munionsubadditive}) and Lemma \ref {MhatMagree}.
By Lemma \ref {MhatMagree}, it is sufficient to prove the result in the case where $A = \hat A$ is the union of squares in ${\cal S}$.

For each $j$ in $\Z$, we will call $\mathcal S_j$ the set of squares that are at height between $2^{j}$ and $2^{j+1}$.
We will say that a square $S = [a 2^j, (a+1) 2^j] \times [2^j, 2^{j+1}]$ in $\mathcal S_j$ is even (respectively odd) if $a$ is even (resp. odd).
We know that
\begin {equation}
\label {eqsum}
\sum_{S \in {\cal S}(A) } M (S) \le  \lim_{s \to \infty} s \E \left( \sum_{S \in {\cal S} (A) } (\Im B^{is}_{\tau(S)})^\alpha \right).
\end {equation}
To bound this expectation, we note that for each $j \in \mathbb Z$, for each even square $S \in \mathcal S_j$, and for each $z \in S$, the probability that a Brownian motion started from $z$
hits the real line before hitting any other even square in $\mathcal S_j$ is bounded from below independently from $z$, $j$ and $S$.
Hence, the strong Markov property implies that the total number of even squares in $\mathcal S_j$ hit by a Brownian motion before hitting the real line is stochastically dominated by a geometric random variable with finite universal mean (independently of the starting point of the Brownian motion) that we call $K/2$. The same is true for odd squares.

Note also that if the starting point $z$ of $B^z$ is in $S \in {\mathcal S}_j$ and if $k \ge 1$, then the probability that the imaginary part of $B$ reaches $2^{j+1+k}$ before $B$ hits the real line is not larger than $2^{-k}$. It follows from the strong Markov property that the expected number of squares in ${\mathcal S}_{j+k+1}$ that $B$ visits before exiting
$\H$ is bounded by $K\times 2^{-k}$.
It follows that for a Brownian motion started from $z$ with $2^j \le \Im (z) \le 2^{j+1}$
$$
\E \left( \sum_{S \in {\cal S} } (\Im B^z_{\tau(S)})^\alpha \right)
 \le
 \sum_{k \le 1} K (2^{j+k+1})^\alpha  +  \sum_{k \ge 2} K 2^{-k} (2^{j+k+1})^\alpha
\le  C 2^{j \alpha} \le C (\Im z)^\alpha
$$
for some universal constant $C$ (bear in mind that $\alpha < 1$ and that for $S \in {\cal S}_{j+k}$, $(\Im B^x_{\tau (S)}) \le 2^{(j+k+1)}$).
If we now apply this statement to the Brownian motion $B^{is}$ after its first hitting time $\tau (A)$ of $A \cup \R = \hat A \cup \R$, we get that for all large $s$ and for some universal positive constant $C'$,
$$
\E \left( ( \Im B^{is}_{\tau (A)})^\alpha \right)
\ge
C'
\E \left( \sum_{S \in {\cal S} (A) } ( \Im B^{is}_{\tau(S)})^\alpha \right) .
$$
Combining this with (\ref {eqsum}) concludes the proof.
\end{proof}

For each square $S = [a 2^j, (a+1) 2^j] \times [2^j, 2^{j+1}]$ of $\mathcal S$, we can define the union $R(S)$ of $S$ with all the squares of ${\mathcal S}$ that lie strictly {\em under} $S$, i.e.
$R(S) = [a 2^j, (a+1) 2^j] \times [0, 2^{j+1}]$.
Note that that scaling shows immediately that for some universal constant $C''$ and for all $S \in \mathcal S$,
 \begin{equation} \label{MR} M(R(S)) = C'' M(S). \end{equation}

\subsection {Estimates for loop-soup clusters}
Let us now use these quantities to study our random loop-ensembles.
Suppose that $\overline \Gamma$ is the conformal loop ensemble corresponding to any given $c \in (0,c_0]$.
Given a hull $A$ we denote by $\tilde A = \tilde A ( A, \overline \Gamma)$ the random hull whose complement is the unbounded component of the set obtained by removing from $\H \setminus A$ all the loops of $\overline \Gamma$ that intersect $A$.  Local finiteness implies that $\tilde A$ is itself
a hull almost surely.

\begin {figure}[htbp]
\begin {center}
\includegraphics [width=6in]{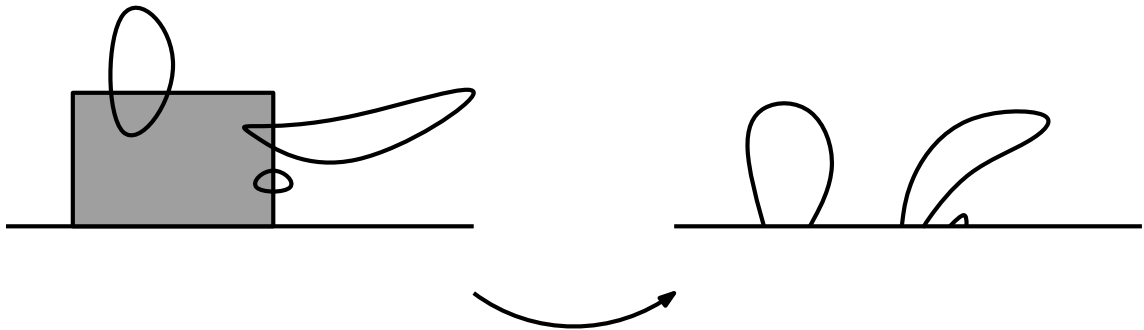}
\caption {Construction of $\Phi_A ( \tilde A)$ (sketch)}
\end {center}
\end {figure}
Now define
$$
N(A) = N_\kappa(A) := \E (M(\Phi_A(\tilde A))).
$$
Recall that if $c < 1$ and $c \le c_0$ then $\mathbb P_c$ defines a CLE$_\kappa$ with $\kappa < 4$. We can therefore reformulate in terms of $c$ a proposition
of \cite{sheffieldwerner1} as follows (this corresponds intuitively to the statement that SLE is unlikely to be very close to a boundary arc when $\kappa < 4$):
\begin{proposition}\cite {sheffieldwerner1} \label{p.swbound}
If $c <1$ and $c \le c_0$, then there is an $\alpha(c)  \in (0,1)$ such that, if we denote by $\diam(A)$ the diameter of
$A$, then we have $\E ( \diam( \tilde A)^{1+\alpha}) < \infty$ for all hulls $A$.
\end{proposition}
Throughout the remainder of this subsection, we will suppose that $c_1 < 1$ and $c_1 \le c_0$, and that this $c_1$ is fixed. We then choose $\alpha = \alpha (c_1)$, and we define $M$ and $N$ using this value of $\alpha$.
We will then let $c$ vary in $[0, c_1]$. It follows from the previous proposition that for all $c \le c_1$,
$$N(A) = \E ( M ( \Phi_A (\tilde A))) \le \E ( M ( \tilde A)) \le \E ( \diam(\tilde A)^{1+\alpha}) < \infty .$$
Here is a more elaborate consequence of the previous proposition:
\begin{corollary} \label{rectangleNbound}
Consider $c \leq c_1$ and $\alpha = \alpha(c_1)$ fixed as above.  For any hull $A$ and any $S \in \mathcal S$, if $A_S = A \cap S$, then
$$\E( M ( \Phi_A (\tilde A_S))) \le C(c) M(S),$$ for some constant $C(c)$ depending only on $c$ and
tending to zero as $c \to 0$.
\end{corollary}
\begin{proof}

By scaling, it suffices to consider the case where $S=[0,1] \times [1,2]$ and hence $R=R(S)$ is the rectangle $[0,1] \times [0,2]$.
Proposition \ref{p.swbound} then implies that
$$
\E (M(\Phi_A (\tilde A_S))) \le \E ( M (\tilde A_S )) \le \E ( M ( \tilde R)) \le \E ( \diam( \tilde R)^{1+\alpha} ) < \infty.
$$
We want to prove that
$\E (M(\Phi_A (\tilde A_S)))$ tends to zero uniformly with respect to $A$ as $c \to 0$.
Let $E(c)$ denote the event that some loop-soup cluster (in the loop-soup of intensity $c$) intersecting the rectangle $R$ has radius more than $\epsilon^2$.
When $E(c)$ does not hold, then standard distortion estimates yield an $\epsilon$ bound on the height of (i.e., the largest imaginary part of an element of) $\Phi_A (\tilde A_S)$. But we then also know that $\tilde A_S$ is a subset of $[-1, 3] \times [0,3]$, so that a Brownian motion started from $is$ will hit $\tilde A_S$ before hitting $A \cup \R$ with a probability bounded by $s^{-1}$ times some universal constant $C$. Hence, unless $E(c)$ holds, we have $M ( \Phi_A ( \tilde A_S )) \le C \eps^\alpha$.

Summing up, we get that
$$ \E ( M ( \Phi_A ( \tilde A_S))) \le C \eps^\alpha + \E ( 1_{E(c)} M ( \tilde A_S))
\le C \eps^\alpha + \E ( 1_{E(c)} \diam( \tilde R^{c_1})^{1+\alpha} ), $$
 where $\tilde R^{c_1}$ denotes the $\tilde R$ corresponding to a larger loop-soup of intensity $c_1$ that we couple to the loop-soup of intensity $c$.
But $\P(E(c)) \to 0$ as $c \to 0$ and $\E ( \diam( \tilde R^{c_1})^{1+\alpha} ) < \infty$, so that if we take $c$ sufficiently small,
$$ \E ( M ( \Phi_A ( \tilde A_S))) \le  2 C \eps^\alpha $$
for all hulls $A$. This completes our proof.
\end{proof}

We are now ready to prove our final Lemma:
\begin{lemma} \label{generalNbound}
For $c \le c_1$, there exists a finite constant $C_1 = C_1 (c)$ such that for
all hulls $A$, $N(A) \leq C_1(c) M(A)$. Furthermore, we can take $C_1(c)$ in such a way that
$C_1(c)$ tends to zero as $c \to 0$.
\end{lemma}
\begin{proof}
Putting together the estimates in
Lemma \ref{mhatsum} and Corollary \ref {rectangleNbound}
we have
\begin {eqnarray*}
 N(A)  & = &  \E ( M ( \Phi_A ( \tilde A))) \\
 & =  & \E ( M ( \Phi_A (\underline \cup_{S \in {\cal S}(A)} \tilde A_S )))
\\
& = & \E ( M ( \underline \cup_{S \in {\cal S}(A)} \Phi_A (\tilde A_S)))  \\
& \le &  \E ( \sum_{S \in {\cal S} (A)} M ( \Phi_A (\tilde A_S)))\\
&  \le&   \sum_{S \in {\cal S}(A)}  C(c) M(S) \\
& \le &  C(c) (C')^{-1} M(A)
\end {eqnarray*}
and $C(c) \to 0$ when $c \to 0+$, whereas $C'$ does not depend on $c$.
\end{proof}

As we will now see, this property implies that for $c = c_0$, $N$ is necessarily infinite for all positive $\alpha$ (i.e. it shows that the size of clusters at the critical point can not decay too fast), and this will enable us to conclude the proof of Proposition \ref{criticaltheorem} in the manner outlined after its statement:

\begin {proof}
Suppose that $c_0 < 1$. We choose $c_1=c_0$ (and $\alpha=\alpha (c_1)$). We take $c'$ to be positive but small enough
so that the product of the corresponding constants $C_1(c_0)$ and $C_1 (c')$ in Lemma \ref{generalNbound} is less than $1$.
We will view the loop-soup $\Gamma$ with intensity $c_0 + c'$ as the superposition of a loop-soup $\Gamma_0$ with intensity $c_0$ and an independent loop-soup
$\Gamma'$ with intensity $c'$ i.e. we will construct $\overline \Gamma$ with
 via the loop-soup cluster boundaries in $\overline \Gamma_0$ and $\overline \Gamma'$.

Now let us begin with a given hull $A$ (say the semi-disk of radius $1$ around the origin).
Suppose that $\Gamma$ contains a chain of loops that join $A$ to the line $L_R= \{ z \in \H \ : \  \Im (z) =R  \}$. This implies that one can find a finite chain $\overline \gamma_1, \ldots, \overline \gamma_n$ (chain means that two consecutive loops intersect) of loops
in $\overline \Gamma_0 \cup \overline \Gamma'$ with $\overline \gamma_1 \cap A \not= \emptyset$ and $\overline \gamma_n \cap L_R \not= \emptyset$.
Since the loops in $\overline \Gamma_0$ (resp. $\overline \Gamma'$) are disjoint, it follows that the loops $\overline \gamma_1, \ldots , \overline \gamma_n$
alternatively belong to $\overline \Gamma_0$ and $\overline \Gamma'$.

Consider the loops of $\overline \Gamma_0$ that intersect $A$. Let us consider $A_1$ the hull generated by the union of  $A$ with these loops (this is the $\tilde A$ associated to the loop-soup $\Gamma_0$). Recall that the expected value of $M( A_1)$ is finite because $\alpha= \alpha( c_1)$.
Then add to $A_1$, the loops of $\overline \Gamma'$ that intersect $A_1$. This generates a hull $B_2$ (which is the $\tilde A_1$ associated to the loop-soup $\Gamma'$).
Then, add to $B_2$ the loops of $\overline \Gamma_0$ that intersect $B_2$. Note that in fact, one basically adds only the loops of $\overline \Gamma_0$ that intersect
$A_1 \setminus A$ (the other ones were already in $A_1$) in order to define a new hull $B_3$, and continue iteratively.
Let $F$ be the limiting set obtained. We can also describe this sort of exploration
by writing for all $n \ge 1$, $A_{n+1} = \Phi_{A_n}(\tilde A_n)$, where $\tilde A_n$ is alternately constructed from $A_n$ using a loop-soup with intensity $c_0$  or
$c'$ as $n$ is even or odd.
The expected value of $M(A_n)$ decays exponentially, which implies (Borel-Cantelli) that $M(A_n)$ almost surely decays eventually faster than some exponential
sequence.

We note that if $A$ is a hull such that for all $z \in A$, $\Im (z) \le 1$, we clearly have $\hcap (A) \le M(A)$. On the other hand, we know that if $A$ is a hull such that
there exists $z \in A$ with $\Im (z) \ge 1$, then $M(A) \ge c$ for some absolute constant $c$. Hence, we see that almost surely, for all large enough $n$, $\hcap (A_n) \le M(A_n)$, which implies that almost surely
 $\sum_n \hcap(A_n) < \infty$. But the half-plane capacity behaves additively under conformal iterations, so that in fact
$ \hcap (F) = \sum_{n \ge 0} \hcap (A_n)$.
Hence, for large enough $R$, the probability that $F$ does not intersect $L_R$ is positive, and
 there is a positive probability that no chain of loops in $\Gamma$ joins $A$ to $L_R$.
It follows that $\mathbb P_{c_0 +c'} ( \overline {\mathcal C} =  \{ \overline \HH \} ) < 1$ and
 Proposition~\ref{p.satisfiesconformal} would then imply that $c_0  +c' \le c_0$ which is impossible.
This therefore implies that $c_0 \ge 1$. As explained after the proposition statement, we also know that $c_0$ can not be strictly larger than $1$, so that we can finally
conclude that $c_0= 1$.
\end {proof}

As explained at the beginning of this section, this completes the proof of our main results, Theorems \ref{loopsoupsatisfiesaxioms} and
\ref{kappacorrespondence}.

\medbreak

\noindent\begin{tabular}{lll}
{\it Department of Mathematics MIT, 2-180} & & {\it Laboratoire de Math\'ematiques}\\
{\it 77 Massachusetts Avenue} & & {\it Universit\'e Paris-Sud 11, b\^{a}timent 425}\\
{\it Cambridge, MA 02139-4307} & & {\it 91405 Orsay Cedex, France} \\
sheffield@math.mit.edu & & wendelin.werner@math.u-psud.fr
         \end{tabular}

\end{document}